\documentclass[10pt]{amsart}

 \usepackage{amssymb,amsthm,amsfonts,amsmath}
\newcommand{\ii}{\mathrm{i}}
\newtheorem{thm}{Theorem}

\newtheorem{lem}[thm]{Lemma}

\newtheorem{prop}[thm]{Proposition}

\theoremstyle{definition}
\newtheorem{defn}[thm]{Definition}

\newtheorem{rem}[thm]{Remark}
\newtheorem{cor}[thm]{Corollary}

\newcommand{\cR}{\mathcal{R}}

\newcommand{\cD}{\mathcal{D}\,}
\newcommand{\cH}{\mathcal{H}\,}

\newcommand{\ov}{\overline}

\newcommand{\cN}{\mathcal{N}}

\usepackage{color}

\begin{document}
\begin{title}
{Domain Characterizations of Strong  Commutativity\break of Unbounded Self-Adjoint  Operators}
\end{title}
\author{Konrad Schm\"udgen}
\address{University of Leipzig, Mathematical Institute, Augustusplatz 10, D-04109 Leipzig, Germany}
\email{schmuedgen@math.uni-leipzig.de}
\maketitle

\begin{abstract}
Let $A$ and $B$ be unbounded self-adjoint operators on a Hilbert space. Suppose  that there exists a dense linear subspace  $\cD\subseteq \cD(AB)\cap \cD(BA)$ such that $ABx=BAx$ for $x \in \cD$. We give some characterizations of the strong commutativity of $A$ and $B$ in terms of domain equalities.
\end{abstract}

\bigskip


\textbf{Key  words:} commuting self-adjoint operators, strong commutativity

\bigskip

\section{Introduction and Main Result}

For unbounded self-adjoint operators on Hilbert space there are  various versions of commutativity. In this paper we will use the following two notions. 

 \begin{defn}\label{defcommute}
Let $A$ and $B$ be  self-adjoint operators on a Hilbert space $\cH$.
We shall say that $A$ and $B$

$\bullet$  \emph{commute} if there exists a dense linear subspace $\cD\subseteq \cD(AB)\cap \cD(BA)$ such that
\begin{align*}ABx=BAx \quad \textrm{for}~~~x \in \cD,
\end{align*}

$\bullet$ \emph{strongly commute} if the spectral measures  $E_A$ and 
$E_B$  of $A$ and $B$ commute, that is, 
\begin{align*} E_A(M)E_B(N)=E_B(N)E_A(M)\quad \textrm{for all Borel subsets}~~~  M,N~~\textrm{of}~~(-\infty,+\infty).
\end{align*}
\end{defn}

The first notion is much weaker than the second. In general, one cannot recover unbounded self-adjoint operators $A$ and $B$ from their restrictions to the dense domain $\cD$. It might be natural to require in addition that $\cD$ is a core for $A$ and $B$. But even then it does not follow   that $A$ and $B$ strongly commute. As first noted by E. Nelson \cite{nelson} there exist self-adjoint operators which  commute on a common core, but   do not strongly commute. Examples of this kind are studied in \cite{fuglede}, \cite{schm84}, \cite{sf}, \cite{Mo}.

In contrast, for strongly commuting self-adjoint operators there is a well-developed theory based on the joint spectral theorem (see e.g. \cite[Theorem 5.23]{schm12}). Clearly, if $A$ and $B$ strongly commute, they also commute  (\cite[Corollary 5.28]{schm12}). Note that two self-adjoint operators $A$ and $B$ commute strongly if and only if the resolvents $(A-\alpha I)^{-1}$ and $(B-\beta I)^{-1}$ commute for some, and then for all,  $\alpha\in \rho(A)$ and $\beta \in \rho(B)$.

The aim of this  note is to derive some new criteria  for the strong commutativity of commuting self-adjoint operators which are based on  domain equalities.

The following theorem is the main result of this paper. Statements (ii)--(v) provide characterizations of the strong commutativity of two commuting self-adjoint operators $A$ and $B$ in terms of domains.

\begin{thm}\label{mainth}
Suppose that $A$ and $B$ are commuting self-adjoint operators on $\cH$. Let $\alpha, \beta, \gamma\in {\rm C}$. Suppose  that $\alpha\in \rho(A)$ and $\beta\in \rho(B)$. Then the following statements are equivalent:
\begin{itemize}
\item[\em (i)] ~ $A$ and $B$ commute strongly.
\item[\em (ii)] ~ $\cD\big((A+\alpha I)(B+\beta I)\big)=\cD\big((B+\beta I)(A+\alpha I)\big).$
\item[\em (iii)] ~ $\cD\big((A+\alpha I)(B+\beta I)\big)\subseteq \cD\big((B+\beta I)(A+\alpha I)\big).$
\item[\em (iv)]~ $\cD\big(\, \ov{(A-\gamma I)(B-\beta I)}\, \big)=\cD((B-\beta I)(A-\gamma I)).$
\item[\em (v)]~ $\cD((A-\gamma I)(B-\beta I))\subseteq\cD((B-\beta I)(A-\gamma I)).$
\end{itemize}
\end{thm}

 Note that at least one of the numbers $\alpha, \beta, \gamma$ in conditions (ii)--(v) belongs to the resolvent set $\rho(A)$ or $ \rho(B)$. If both numbers are in the spectra of $A$ and $B$,  one cannot conclude the strong commutativity from the corresponding domain equalities, as discussed in Case 3 of  Remark \ref{Remark} below. \smallskip

The proof of Theorem \ref{mainth} will be given in Section \ref{proofmth}.
In Sections \ref{domd1ab} and \ref{productclosed} we develop some preliminaries that will be used in the proof of Theorem \ref{mainth}. Some results of Sections \ref{domd1ab}  are  special cases of  a more general treatment  in \cite{sf}
and they are  of interest in itself for the commutation problem.

All operator-theoretic notions and results used in this paper are standard; they can be found (for instance) in \cite{schm12}. The domain of an operator $T$ is denoted by $\cD(T)$ and $\rho(T)$ denotes the resolvent set of $T$.

\section{The Domain $\cD_1(A,B)$}\label{domd1ab}

In this section, $A$ and $B$ are self-adjoint operators on a Hilbert space $\cH$. Further,   $\alpha \in \rho(A)$ and $\beta\in \rho(B)$ are fixed numbers.

Let $P_{\alpha,\beta}$ the orthogonal projection on the closure of the range of the commutator of the resolvents  $R_\alpha(A)$ and $R_\beta(B))$, that is, $P_{\alpha,\beta}$ is the projection on the closure of $[(R_\alpha(A),R_\beta(B)]\cH. $ \smallskip

\begin{defn}\label{def2}~ $\cD_1(A,B):=R_\alpha(A)R_\beta(B)(I-P_{\ov{\alpha},\ov{\beta}})\cH.$
\end{defn}
\begin{prop}\label{basicd1ab} Then we have:\\(i) $R_\alpha(A)R_\beta(B)x=R_\beta(B)R_\alpha(A)x$ for $x\in(I-P_{\ov{\alpha},\ov{\beta}}) \cH.$\\
(ii)  $ \cD_1(A,B)\subseteq \cD(AB)\cap \cD(BA)$ and $ABx=BAx$ for $x\in \cD_1(A,B)$,\\
(iii) If $\cD$ is a linear subspace of $  \cD(AB)\cap \cD(BA)$ such that $ABx=BAx$ for $x\in \cD$, then $\cD\subseteq \cD_1(A,B)$.
\end{prop}
\begin{proof}
(i): Since  $ [(R_{\ov{\alpha}}(A),R_{\ov{\beta}}(B)]\cH\subseteq P_{\ov{\alpha},\ov{\beta}}\cH$ by the definition of $P_{\ov{\alpha},\ov{\beta}}$, we have
 $T:=(I-P_{\ov{\alpha},\ov{\beta}})[(R_{\ov{\alpha}}(A),R_{\ov{\beta}}(B)]\equiv 0$. Hence $T^*=-[(R_{\alpha}(A),R_{\beta}(B)](I-P_{\ov{\alpha},\ov{\beta}})=0$, which implies the assertion of (i). \smallskip

(ii): Let $x:=R_\alpha(A)R_\beta(B)u\in \cD_1(A,B)$, where $u\in (I-P_{\ov{\alpha},\ov{\beta}})\cH$. Then $x=R_\beta(B)R_\alpha(A)u$ by (i). Then 
$(B-\beta I)(A-\alpha I)x=u=(A-\alpha I)(B-\beta I)x$. Hence $x\in \cD(BA)\cap \cD(AB)$ and $BAx=ABx$.
\smallskip

(iii):  Let $x\in  \cD\subseteq \cD(AB)\cap \cD(BA)$. Then $x\in \cD(B-\beta I)$ and hence $x=R_\beta(B)y$ for some $y\in \cH$. Further, since $x\in \cD(A)\cap \cD(AB)$, we have $y=(B-\beta I)x\in \cD(A)$, so that $y=R_\alpha(A)u$ for some $u\in \cH$. Thus  $x= R_\beta(B)R_\alpha(A)u$. Similarly,  $x=R_\alpha(A)R_\beta(B)v$, with $v\in \cH$. 

Using $ABx=BAx$ we obtain $u=(B-\beta I)(A-\alpha I)= (A-\alpha I)(B-\beta I)x=v$, that is $u=v$. Therefore, $[R_\alpha(A),R_\beta(B)]u=0$. Since $u\in \cN([R_\alpha(A),R_\beta(B)])$, we have 
$u\bot \cR([R_\alpha(A),R_\beta(B)]^*)=\cR([R_{\ov{\alpha}}(A),R_{\ov{\beta}}(B)])=P_{\ov{\alpha},\ov{\beta}}\cH$. That is, $x=[(R_\alpha(A),R_\beta(B)]u$ with $u\in (I-P_{\ov{\alpha},\ov{\beta}})\cH$, which means that $x\in \cD_1(A,B)$.
\end{proof}

Proposition \ref{basicd1ab} has a number of interesting consequences. It shows that $D_1(A,B)$ is the largest linear subspace $\cD$ of $\cD(AB)\cap \cD(BA)$ for which $ABx=BAx$ for all $x\in \cD$.  From this characteritation it follows in particular that $D_1(A,B)$ is independent on the particular numbers $\alpha\in\rho(A)$ and $\beta\in \rho(B)$ entering into Definition \ref{def2}. Further, Proposition \ref{basicd1ab} gives at once  the following corollary.

\begin{cor}\label{cordense}
Two self-adjoint operators $A$ and $B$ on $\cH$ commute (according to Definition \ref{defcommute}) if and only if $\cD_1(A,B)$ is dense in $\cH$.
\end{cor}

Another interesting fact is the following.
\begin{prop}\label{dimab}  $\dim P_{\alpha,\beta}\cH=\dim P_{\gamma,\delta}\cH$ for $\alpha,\gamma\in\rho(A)$ and $\beta,\delta\in \rho(B)$. 
\end{prop}
\begin{proof} Recall that   $$\cD_1(A,B)=R_\alpha(A)R_\beta(B)(I-P_{\ov{\alpha},\ov{\beta}})\cH=R_\alpha(A)R_\delta(B)(I-P_{\ov{\alpha},\ov{\delta}})\cH.$$
Using that  $\cN(R_\alpha(A))=\{0\}$ we obtain 
$$
(B-\delta I)R_\beta (B)(I-P_{\ov{\alpha}, \ov{\beta}})\cH= (I-P_{\ov{\alpha},\ov{\delta}})\cH.
$$
Since $\beta,\delta\in \rho(B)$, the operator  $(B-\delta I)R_\beta(B)$ is an isomorphism of the underlying Hilbert space $\cH$. Hence we conclude that 
${\rm codim}(I-P_{\ov{\alpha}, \ov{\beta}})\cH ={\rm codim}(I-P_{\ov{\alpha},\ov{\delta}})\cH$, so that
\begin{align}\label{dimabcd}
\dim P_{\ov{\alpha}, \ov{\beta}}\cH =\dim P_{\ov{\alpha},\ov{\delta}}\cH.
\end{align}
By Proposition \ref{basicd1ab}(i) we have also  $$\cD_1(A,B)=R_\delta(B)R_\alpha(A)(I-P_{\ov{\alpha},\ov{\delta}})\cH=R_\delta(B)R_\gamma(A)(I-P_{\ov{\gamma},\ov{\delta}})\cH.$$
Then the  same reasoning gives
$\dim P_{\ov{\alpha}, \ov{\delta}}\cH =\dim P_{\ov{\gamma},\ov{\delta}}\cH.$ Combined with (\ref{dimabcd}) we get
$\dim P_{\ov{\alpha}, \ov{\beta}}\cH =\dim P_{\ov{\gamma},\ov{\delta}}\cH.$ Replacing $\alpha,\beta,\gamma,\delta$ by their complex conjugates yields the assertion.
\end{proof}
Proposition \ref{dimab} justifies the following definition.
\begin{defn}
$d(A,B):= \dim P_{\alpha,\beta}\cH$ for $\alpha\in \rho(A)$,  $\beta \in \rho(B)$ is called the \emph{defect number} of the pair $\{A,B\}$ of self-adjoint operators on $\cH$.
\end{defn}
Clearly, $d(A,B)=0$ if and only if $  P_{\alpha,\beta}=0$, or equivalently,  $A$ and $B$ strongly commute. 

\begin{rem}\label{Remark}
In this Remark we discuss Theorem \ref{mainth} in the case  $\alpha=\beta =\gamma=0$. Assume throughout that $A$ and $B$ are commuting self-adjoint operators on $\cH$.\smallskip

Case 1:  $0\in \rho(A)\cap \rho(B)$.\\  Then condition (ii) applies. Hence $A$ and $B$ strongly commute if and only if $\cD(AB)=\cD(BA)$. 
Further, by \cite[Proposition 1.11]{sf}, $AB\lceil \cD_1(A,B)$ is a densely defined closed symmetric operator with deficiency indices $(d(A,B),d(A,B))$. This gives a nice interpretation of the defect number $d(A,B)$.

Case 2:  $0\in \rho(B)$. \\ 
Then  one can use condition (iii), so $A$ and $B$ strongly commute if and only if $\cD(\ov{AB})=\cD(BA)$. 

Case 3: Suppose that  $AB\lceil \cD_1(A,B)$ is essentially self-adjoint.\\Then it follows easily that $\ov{AB}=\ov{BA}$ and this operator is self-adjoint. However, if $0$ is neither in $ \rho(A)$ nor in $\rho(B)$, one cannot conclude that $A$ and $B$ strongly commute. For instance, for the example developed in \cite{fuglede} the operator $AB\lceil \cD_1(A,B)$  is essentially self-adjoint  \cite[Theorem 2]{fuglede} and $A$ and $B$ do not commute strongly.  It can be shown that this pair has defect number $d(A,B)=1$. Pairs of self-adjoint operators with defect number one are extensively studied in \cite{sf}.
\end{rem} 
\section{Closedness of the product of closed operators}\label{productclosed}

\begin{prop}\label{lemmaclosedop}
Let $T$ and $S$ be closed linear operators on a Hilbert space $\cH$. Then the operator $TS$ is closed if and only if   there exists a constant $c>0$ such that
\begin{align}\label{aabinequality}
\|Sx\|\leq c~ (\|TSx\|+ \|x\|)\quad \textrm{for all}~~~ x\in \cD(TS).
\end{align}
\end{prop}
\begin{proof}   First we suppose that condition (\ref{aabinequality}) is satisfied. Let $(x_n)_{n\in N}$ be a sequence of vectors $x_n\in \cD(TS)$ such that the sequences  $(x_n)_{n\in N}$ and $(TSx_n)_{n\in N}$ converge in $\cH$.  Let $u:=\lim_n x_n$ and $v:=\lim_n TSx_n$. From (\ref{aabinequality}), applied to $x=x_n-x_k$,  it follows that $(Sx_n)_{n\in N}$ is a Cauchy sequence. Hence it converges. Since the operator $S$ is closed, we conclude that $Su:=\lim_n Sx_n$. Using that $T$ is closed we obtain  $TSu:=\lim_n TSx_n$, which proves that $TS$ is closed.

Now we assume that the operator $TS$ is closed. Then the domain $\cD(TS)$ is a Hilbert space, denoted $\cH_{TS}$, equipped with the scalar product
$$
\langle x,y\rangle_{TS}:=\langle TSx,TSy\rangle +\langle x,y\rangle,~~~ x,y \in \cD(TS).
$$ 
We show that the operator $S_0:=S\lceil \cD(TS)$  of the Hilbert space $\cH_{TS}$ into the Hilbert space $\cH$  is closed. For suppose that  $(x_n)_{n\in N}$ be a sequence with $x_n\in \cD(TS)$ 
which converges to $x$ in $\cH_{TS}$ such that  $(S_0x_n)_{n\in N}$ converges to $y$ in $\cH$. In particular,  $x\in \cD(TS)$ and $x=\lim_n x_n$  in $\cH$. Therefore, since the operator $S$ on $\cH$ is closed,  $Sx=y$. From $x\in \cD(TS)$ we get $S_0x=Sx=y$. This shows that the mapping $S_0:\cH_{TS}\to \cH$ is closed. By the closed graph theorem,  $S_0:=S\lceil \cD(TS)$ is continous in the corresponding Hilbert space norms. This means that there is a constant $c>0$ such that  (\ref{aabinequality}) holds.
\end{proof}

\section{Proof of the Main Theorem}\label{proofmth}

The following  lemma is the crucial technical ingredient for the proofs of our main results.
\begin{lem}\label{crucial} Let $A$ and $B$ be commuting self-adjoint operators on $\cH$. Let $\alpha\in \rho(A)$,  $\beta\in \rho(B)$ and   $x,y\in \cH$. If  $R_\alpha(A)R_\beta(B)y=R_\beta(B)R_\alpha(A)x,$ then $y=x$.
\end{lem}
\begin{proof} Let $u\in \cH$. Using Proposition \ref{basicd1ab}(i) we derive
\begin{align}
&\langle y, R_{\ov{\alpha}} (A)R_{\ov{\beta}}(B) (I-P_{\alpha, \beta})u\rangle\label{comy} \\ &= \langle y, R_{\ov{\beta}}(B)R_{\ov{\alpha}} (A) (I-P_{\alpha, \beta})u\rangle\nonumber \\ &=\langle R_\alpha(A)R_\beta(B)y, (I-P_{\alpha, \beta})u \rangle\nonumber \\  & =\langle R_\beta(B)R_\alpha(A)x,(I-P_{\alpha, \beta})u \rangle\nonumber \\ &=\langle x, R_{\ov{\alpha}} (A)R_{\ov{\beta}}(B) (I-P_{\alpha, \beta})u\rangle. \label{comx}
\end{align}
Since $A$ and $B$ commute, $\cD_1(A,B) = R_{\ov{\alpha}} (A)R_{\ov{\beta}}(B) (I-P_{\alpha, \beta})\cH$ is dense by Corollary \ref{cordense}. Therefore, comparing  (\ref{comy}) and (\ref{comx}) it follows that $y=x$. 
\end{proof}

Now we begin the proof of Theorem \ref{mainth}. Clearly, 
\begin{align}\label{domequ1} 
\cD((A-\alpha I)(B-\beta I))&=\cR(R_\beta(B)R_\alpha(A)),\\  \cD((B-\beta I)(A-\alpha I))&=\cR(R_\alpha(A)R_\beta(B)).\label{domequ2}
\end{align}

(i)$\to$(ii): Since $A$ and $B$ strongly commute, the resolvents $R_\alpha(A)$ and $R_\beta(B)$ commute.  Hence the ranges  $R_\beta(B)R_\alpha(A)\cH$ and $R_\alpha(A)R_\beta(B)\cH$ coincide and so do the domains $\cD((A-\alpha I)(B-\beta I))$ and $\cD((B-\beta I)(A-\alpha I))$ by  (\ref{domequ1}) and  (\ref{domequ2}). \smallskip

(iii)$\to$(i):
 Let $x\in \cH$. Since $R_\beta(B)R_\alpha(A))x \in\cD((A-\alpha I)(B-\beta I))$ by (\ref{domequ1}), it follows from  (iii) and (\ref{domequ2}) that there exists $y\in \cH$ such that $R_\beta(B)R_\alpha(A))x=R_\alpha(A)R_\beta(B)y$. Therefore  Lemma \ref{crucial} implies that $x=y$. Then $R_\beta(B)R_\alpha(A))x=R_\alpha(A)R_\beta(B)x$. Thus the resolvents $R_\alpha(A)$ and $R_\beta(B)$ commute. Hence $A$ and $B$ strongly commute.
\smallskip

(i)$\to$(iv): By the spectral theorem \cite[Theorem 5.23]{schm12} for the strongly commuting self-adjoint operators $A$ and $B$ there exists a joint spectral measure $E$ on $R^2$ such that 
\begin{align}
A=\int_{R^2} t_1 \, dE(t_1,t_2)\quad \textrm{and}\quad  B=\int_{R^2} t_2 \, dE(t_1,t_2).
\end{align}
Then, setting  $f_1(t_1,t_2)=t_1-\gamma$ and $f_2(t_1,t_2)=t_2-\beta$, the operators $A-\gamma I$ and $B-\beta I$ are the spectral integrals $I(f_1)=\int f_1 dE$ and $I(f_2)=\int f_2 dE$, respectively. Then $I(f_1f_2)=\ov{I(f_1)I(f_2)}=\ov{I(f_2)I(f_1)}$ by the product formula of the functional calculus \cite[Theorem 4.16(iii)]{schm12}. This yields 
\begin{align}\label{closab}\ov{(A-\gamma I)(B-\beta I)}= \ov{(B-\beta I)(A-\gamma I)}.
\end{align}
Since $\beta \in \rho(B)$, there is a constant $c>0$ such that $\|(B-\beta I)y\|\geq c\|y\|$ for all $y\in \cD(B)$. This implies that condition (\ref{aabinequality}) is satisfied with $T:=B-\beta I$ and $S:= A-\gamma I$. Therefore, by Proposition \ref{lemmaclosedop}, the operator $TS$ is closed, that is,  $ \ov{(B-\beta I)(A-\gamma I)} =(B-\beta I)(A-\gamma I)$. Inserting this into (\ref{closab}) and considering the domains of both sides we obtain (iv).\smallskip

(v)$\to$(i):
If $\gamma$ is not real, then $\gamma \in \rho(A)$ and the assertion follows from the implication (iii)$\to$(i). If $\gamma $ is real, then $A-\gamma I$ is a self-adjoint operator and we have 
$\cD_1(A-\gamma I,B)=\cD_1(A,B)$. Therefore,  we can assume without loss of generality that $\gamma=0$. Then (v) means that 
\begin{align}\label{domequ11}
\cD(A(B-\beta I))\subseteq\cD((B-\beta I)A).
\end{align}

Let $u\in \cD(A)$ and set $v:=R_\beta(B)u$. Then, since  $(B-\beta I)v=u$, we have $v\in \cD(A(B-\beta I))$. Therefore,   $v\in \cD((B-\beta I)A)$ by  (\ref{domequ11}). In particular, $v\in \cD(A)$, so there exists a vector $w\in \cH$ such that $v=(A-{\rm i} I)^{-1} w$. Then  $Av=w+{\rm i}  (A-\ii I)^{-1} w=(w+\ii v)\in \cD(B-\beta I)$. Since $v$ is also in $\cD(B+\beta I)$, $(A-\ii I)v\in \cD(B-\beta I)$ and hence $(A-\ii I)v=R_\beta (B) y$ for some $y\in \cH$. Then $(B-\beta I)(A-\ii I)v=y$ and hence $v=R_{\ii}(A)R_\beta(B)y$. On the other hand, $(A-\ii I)(B-\beta I)v=(A-\ii Iu$ yields $v= R_\beta (B)R_{\ii}(A)(A-\ii I)u.$ Therefore, applying Lemma \ref{crucial} we obtain $y=(A-\ii I)u.$ Hence
\begin{align}\label{comre1}
R_\beta (B)y= R_\beta (B)(A-\ii I)u.
\end{align}
Further, since  $v=R_\beta (B)u$, we have
\begin{align}\label{comre2}
 R_\beta (B)y=(A-\ii I)v=(A-\ii I)R_\beta (B)u.
 \end{align}
This holds for all $u\in \cD(A)$. Since $A$ is self-adjoint, $(A-\ii I)\cD(A)=\cH.$ Setting $z=(A-\ii I)u$ it follows from (\ref{comre1}) and (\ref{comre2}) that $R_\beta (B)R_{\ii}(A)z=R_\ii (A) R_\beta (B)z$ for all $z\in \cH$. Hence $A$ and $B$ strongly commute, that is, (i) is satisfied.
 \smallskip
 
 Since the implications (ii)$\to$(iii) and (iv)$\to$(v) are trivial, the proof of Theorem \ref{mainth} is complete.

\bigskip

{\bf Acknowledgement:} I want to thank Prof. H. Mortad for an interesting email correspondence that inspired me  to  study the problem treated in this paper.


\bibliographystyle{plain}

\begin{thebibliography}{99}

\bibitem{fuglede} Fuglede, B.: Conditions for two self-adjoint operators to commute or to satisfy the Weyl relation, Math. Scand. {\bf 51}(1982), 163--178.




\bibitem{Mo} Mortad, H.: {\it Counterexamples in Operator Theory}, Birkh\"auser-Verlag, Basel, 2022.



\bibitem{nelson}  Nelson, E.: Analytic vectors, Ann. Math. {\bf 70}(1959), 572--614.


\bibitem{schm84} Schm\"udgen, K.: On commuting unbounded self-adjoint operators. I, Acta Sci. Math. (Szeged)  {\bf 47}(1984), 131--146. 

\bibitem{sf} Schm\"udgen, K. and Friedrich, J.: On commuting unbounded self-adjoint operators. II, J. Integral Equ. Operator Theory {\bf 7}(1984), 815--867. 

\bibitem{schm12} Schm\"udgen, K., \emph{Unbounded Self-adjoint Operators on Hilbert Space}, Graduate Texts in Mathematics {\bf 265}, Springer, Cham, 2012.














\end{thebibliography}

\end{document}